\documentclass[10pt,reqno]{amsart}
\usepackage{subfiles}
\usepackage{mathtools}
\usepackage{ifthen,srcltx}
\usepackage{xcolor}
\usepackage{float}
\usepackage{amssymb,amsmath}
\usepackage{wrapfig}
\usepackage{amscd}
\usepackage{amsrefs}
\usepackage{graphicx}
\usepackage[pdftex,colorlinks]{hyperref}

\def\F{\mathbb{F}}

\def\text#1{\hbox{#1}}
\def\Z{{\mathbb Z}}
\def\R{{\mathbb R}}
\def\N{{\mathbb N}}

\def\C{{\mathbb C}}
\def\Q{{\mathbb Q}}

\def\injrad{{\rm inj rad\,}}

\def\rank{{\rm rank}}

\def\Vol{{\rm Vol}}

\def\SL{{\rm SL}}

\def\N{{\rm N}}

\def\bs{\backslash}

    \DeclareFontFamily{U}{wncy}{}
    \DeclareFontShape{U}{wncy}{m}{n}{<->wncyr10}{}
    \DeclareSymbolFont{mcy}{U}{wncy}{m}{n}
    \DeclareMathSymbol{\Sh}{\mathord}{mcy}{"58}

\newtheorem{theorem}{Theorem}[section]

\newtheorem{lemma}[theorem]{Lemma}
\newtheorem{corollary}[theorem]{Corollary}
\newtheorem{proposition}[theorem]{Proposition}

\newtheorem{remark}[theorem]{Remark}

\def\bs{\backslash}
\newcommand{\showcomments}{yes}
\numberwithin{theorem}{section}
\newsavebox{\commentbox}
{\ifthenelse{\equal{\showcomments}{yes}}%
{\footnotemark
        \begin{lrbox}{\commentbox}
        \begin{minipage}[t]{1.in}\raggedright\sffamily\tiny
        \footnotemark[\arabic{footnote}]}
{\begin{lrbox}{\commentbox}}}
{\ifthenelse{\equal{\showcomments}{yes}}
{\end{minipage}\end{lrbox}\marginpar{\usebox{\commentbox}}}
{\end{lrbox}}}
\title{Growth of mod$-2$ homology in higher rank locally symmetric spaces}

\author{Mikolaj Fraczyk}
  \address{Alfr\'ed R\'enyi Institute of Mathematics,  Re\'altanoda utca 13-15, H-1053, Budapest, Hungary}
  \email{fraczyk@renyi.hu}
\newcommand{\commentmikolaj}[1]

\begin{document}
\begin{abstract}
Let $X$ be a higher rank symmetric space or a Bruhat-Tits building of dimension at least $2$ such that the isometry group of $X$ has property $(T)$. We prove that for every torsion free lattice $\Gamma\subset {\rm Isom} X$ any homology class in $H_1(\Gamma\bs X,\F_2)$ has a representative cycle of total length $o_X(\Vol(\Gamma\bs X))$. As an application we show that $\dim_{\F_2} H_1(\Gamma\bs X,\F_2)=o_X(\Vol(\Gamma\bs X)).$
\end{abstract}
\maketitle
\section{Introduction}
Let $F$ be a local field, that is, $\R,\C$ or a finite extension of $\Q_p$ or $\F_p((t))$. A semisimple algebraic group $G$ over $F$ is called a higher rank group if its $F$-rank is at least $2$. Let $X$ be the symmetric space of $G$ if $F$ is archimedean or the Bruhat-Tits building of $G$ if $F$ is non-archimedean. In this paper we study quantitative properties of the group $H_1(\Gamma\bs X,\F_2)$ as $\Gamma$ varies among  lattices of $G$. Our first result says that if $G$ has property $(T)$ then every homology class in $H_1(\Gamma\bs X,\F_2)$ can be represented by a cycle of total length $o_X(\Vol(\Gamma\bs X)$. 
\begin{theorem}\label{tmainL}
Let $G$ be a higher rank group with property $(T)$. Let $(\Gamma_n)_{n\in \N}$ be a sequence of pairwise non conjugate torsion-free lattices in $G$ and let $\alpha_n\in H_1(\Gamma_n\bs X,\F_2)$ be any sequence of homology classes. Then there exists a sequence of representative cycles $c_n\in Z_1(\Gamma_n\bs X,\F_2)$ with total lengths $\ell(c_n)$ such that $$\lim_{n\to\infty}\frac{\ell(c_n)}{\Vol(\Gamma_n\bs X)}=0.$$
\end{theorem}
The property $(T)$ is used only to ensure that the sequence $(\Gamma_n\bs X)_{n\in \N}$ converges Benjamini-Schramm to $X$ (c.f. \cite{7SamN,LevitGelander}). If $(\Gamma_n\bs X)_{n\in \N}$ converges Benjamini-Schramm to $X$ then the conclusion of Theorem \ref{tmainL} holds for all higher rank groups $G$. To prove Theorem \ref{tmainL} we study the representative cycles of minimal total length. Those cycles need to satisfy very strong geometric constrains that ultimately lead to an upper bound on their length inside the thick part of the ambient space (Proposition \ref{plength}). The quantitative description of the "shortest" representative cycles can be used to bound the size of $H_1(\Gamma\bs X,\F_2)$. We show that the dimension $\dim_{\F_2} H_1(\Gamma\bs X,\F_2)$ grows slower than the volume of $\Gamma\bs X$.
\begin{theorem}\label{tmain}
Let $G$ be a higher rank group with property $(T)$. For any sequence of pairwise non-conjugate, torsion-free lattices $(\Gamma_n)$ in $G$ we have 
$$\lim_{n\to\infty}\frac{\dim_{\F_2} H_1(\Gamma_n \bs X,\F_2)}{\Vol(\Gamma_n\bs X)}=0.$$ 
\end{theorem}
The conclusion holds for any higher rank group $G$ if we assume that the sequence $(\Gamma_n\bs X)_{n\in \N}$ converges Benjamini-Schramm to $X$. It seems that our result is the first one dealing with the growth of mod-$p$ Betti numbers in the non-commensurable setting. We review the state of the literature on the growth of Betti numbers in the next part of the introduction. At this point we would like to stress that the complex of differential forms on $\Gamma\bs X$ cannot be used to compute the mod-$p$ cohomology groups so the analytic methods that work for rational Betti numebrs are not accessible. In most cases the results on the growth of mod-$p$ homology groups of symmetric spaces were proven algebraically, most notably using completed cohomology groups \cite{CalEm, BLLS2014}, or they were deduced from the vanishing of the rank gradient  \cite{AbGeNi2015}. In contrast, our method is geometric in nature and uses the higher rank assumption in a very direct way. The characteristic $2$ of the field $\F_2$ plays an important role in the proof and so far we were unable to extend it to odd primes. It is an interesting question whether the analogue of Theorem \ref{tmainL} holds for $p$ odd.

Let us briefly review what is known and what is expected of the growth of the Betti numbers of locally symmetric spaces. The famous L\"uck approximation theorem \cite{Luck94} asserts that if $Y$ is a finite CW-complex and $(Y_n)_{n\in \N}$ is a sequence of finite Galois (regular) covers of $Y$ such that the intersection of the fundamental groups $\bigcap \pi(Y_n)$ is trivial then 
$$\lim_{n\to\infty}\frac{b_k(Y_n,\Q)}{[\pi(Y):\pi(Y_n)]}=\beta_k^{(2)}(Y),$$ where  $b_k(Y,\Q)=\dim_\Q H_k(Y,\Q)$ is the $k$-th rational Betti number and $\beta_k^{(2)}(Y)$ is the $k$-th $L^2$-Betti number.  In \cite{7Sam} Abert, Bergeron, Biringer, Gelander, Nikolov, Raimbault and Samet proved a far-reaching generalisation of this theorem in the setting of higher rank locally symmetric spaces. They show that for any uniformly discrete\footnote{The uniform discreteness assumption will be lifted in an upcoming work \cite{sequel} by a subset of authors of \cite{7Sam}.} sequence of pairwise non-conjugate lattices $(\Gamma_n)_{n\in \N}$ in a higher rank Lie group $G$ with property (T)   the limit $$\lim_{n\to\infty}\frac{\dim_\C H_k(\Gamma_n\bs X,\C)}{\Vol(\Gamma_n\bs X)}=\beta_k^{(2)}(X).$$ Recently this result has been extended to groups over non-Archimedean local fields \cite{LevitGelander}. Note that the above theorems hold for sequences of non-commensurable lattices as opposed to the original L\"uck approximation theorem where the very setting forces us to consider only commensurable sequences of groups.

Much less is known when we replace the coefficient group by a finite field. One of the first results in this direction is due to Calegari and Emerton \cite{CalEm}. Let $\Gamma$ be a lattice in $\SL(2,\C)$ and let $p$ be a rational prime. Calegari and Emerton define $p$-adic analytic towers of covers of $\Gamma\bs \mathbb H^3$ and study the growth of the first mod-$p$ homology group in such towers. Their results imply that in a $p$-adic analytic tower $(\Gamma_n\bs \mathbb H^3)$ the limit 
$$\lim_{n\to\infty}\frac{\dim_{\F_p} H_1(\Gamma_n\bs X,\F_p) }{\Vol(\Gamma_n\bs X)}$$
always exists. To tackle the problem of mod-$p$ homology growth Calegari and Emerton developed the theory of completed homology and cohomology groups (see \cite{CalEmSurv}). 
Their result was later generalized by Bergeron, Linnell, L{\"u}ck and Sauer to $p$-adic analytic towers of CW-complexes in \cite{BLLS2014}.  In both results we consider only the  chains of normal subgroups which are of $p$-power index in $\Gamma$. In the context of growth of $\F_p$-Betti numbers for amenable fundamental groups the last condition may be relaxed. In \cite{LinLuSa2010} Linnel, L\"uck, and Sauer show that for an amenable group $\Gamma$  and a chain of normal subgroups $\Gamma_n$ the limit
$$\lim_{n\to\infty}\frac{\dim_{\F_p}H_k(\Gamma_n,\F_p)}{[\Gamma:\Gamma_n]}$$ exists and is equal to the Ore dimension of certain homology group (see \cite[Theorem 0.2]{LinLuSa2010}). Still, the setting of the theorem makes sense only for sequences of commensurable groups. 


Let $G$ be a higher rank real semisimple Lie group. Margulis normal subgroup theorem \cite{Margulis} implies that $H_1(\Gamma_n\bs X,\Z)$ is finite. 
In \cite[Conjecture 3]{AbGeNi2015} Ab\'ert, Gelander and Nikolov conjectured that for any sequence of pairwise non-conujugate lattices $(\Gamma_n)_{n\in\N}$ in $G$ we have \begin{equation}\label{conjrank} \lim_{n\to\infty}\frac{d(\Gamma_n)}{\Vol(\Gamma_n\bs X)}=\lim_{n_\to\infty}\frac{\log|H_1(\Gamma_n\bs X,\Z)|}{\Vol(\Gamma_n\bs X)}=0\\
\end{equation} where $d(\Gamma_n)$ is the rank of $\Gamma_n$ i.e. the minimal number of generators of $\Gamma_n$. If the limit on the left-hand side exists it is called the \textbf{ rank gradient} of $(\Gamma_n)$, this notion was introduced by Lackenby in \cite{Lackenby2005}. The conjecture implies the vanishing of normalized mod-$p$ Betti numbers because $d(\Gamma_n)\geq \dim_{\F_p}H_1(\Gamma_n\bs X,\F_p)$. In \cite[Theorem 2]{AbGeNi2015} it was shown that (\ref{conjrank}) holds for all sequences of pairwise different subgroups $(\Gamma_n)\subset \Gamma$ where $\Gamma$ is a right-angled lattice in a higher rank semisimple Lie group. 
The argument in \cite{AbGeNi2015} is based on the relation between the rank gradient and the combinatorial cost (see \cite{Gaboriau2000,Levitt1995} and \cite{AbNi2007}). Combinatorial cost is a powerful tool when one wants to study the rank gradient of a sequence of subgroups of a given group but to the author's knowledge this method is not adapted to handle the non-commensurable sequences of lattices. 

Yet another instance where the growth of the first mod-$p$ homology is known is when $\Gamma$ has the congruence subgroup property and $\Gamma_n$ varies among the subgroups of $\Gamma$. In that case we control the rank by the logarithm of index \cite[Proposition 30]{AbGeNi2015}
$$d(\Gamma_n)\leq C_{\Gamma} \log[\Gamma:\Gamma_n].$$

\subsection{Notation.} If $f,g$ are two functions dependent, among others on a variable $X$ we write $f\ll_X g$ if there exists a constant $C$ dependent on $X$ such that $f\leq Cg$. We will write $f=O_X(g)$ if $|f|\ll_X |g|$ and $f=o_X(g)$ if $\lim \frac{f}{g}=0$ and the speed of convergence depends on $X$. Logarithms are always in base $e$. For a rational prime number $p$ we write $\F_p$ for the field with $p$ elements. Let $x\in \Gamma\bs X$, the injectivity radius $\injrad(x)$ is defined as the supremum of radii $r>0$ such that the open ball of radius $r$ around $x$ lifts to $X$. If $M$ is a subset of $\Gamma\bs X$ and $R>0$ then we distinguish the  $R$-thin part $M_{\leq R}:=\{x\in M\mid \injrad(x)\leq R\}$ and the $R$-thick part $M_{\geq R}:=\{x\in M\mid \injrad(x)\geq R\}.$
\subsection{Idea of the proof.}
Let $(X_n)_{n\in\N}$ be a sequence of lattices in $G$. To simplify the argument let us assume in this sketch that $G$ is a real semisimple Lie group, that the \textbf{fundamental rank} $\delta(G)$ of $G$ is at least $2$ and that the injectivity radius $R_n=\injrad(\Gamma_n\bs X)$ goes to infinity. The fundamental rank is the difference between the absolute rank of $G$ and the maximal rank of a compact torus in $G$, e.g. $\delta(\SL(2,\C))=1,\delta(\SL(2m,\R))=m-1$ and $\delta(\SL(2m+1,\R))=m$. This assumption will ensure that the closed geodesics in $\Gamma\bs X$ lie in closed flats of dimension at least $2$.
The advantage of working with $\F_2$ is that every homology class is represented by a sum of unoriented cycles. By Theorem \ref{tmainL} the elements of $H_1(\Gamma_n \bs X,\F_2)$ can be represented by combinations of closed geodesics whose total length is $o(\Vol(\Gamma_n\bs X)$ (see Proposition \ref{plength}). This is the key ingredient of the proof of Theorem \ref{tmain} . Using a nerve complex constructed by Gelander \cite{Gelander2} and the uniform bound on the lengths of representatives we show that $| H_1(\Gamma_n \bs X,\F_2)|\leq 2^{o(Vol(\Gamma_n\bs X))}$ (see Proposition \ref{p.MTech}) which gives Theorem \ref{tmain}. 

To prove Theorem \ref{tmainL} we define the "reduced representatives" of a homology class in $H_1(\Gamma_n \bs X,\F_2)$. Those are representatives $c\in Z_1(\Gamma_n\bs X,\F_2)$ of minimal total length. Since we do not need to care about the orientation it is easy to see that $c$ is always a sum of uniformly separated closed geodesic (see Lemma \ref{lseparation}). If the fundamental rank is at least $2$ then every closed geodesic on $\Gamma_n\bs X$ is contained in a closed totally geodesic flat subspace of dimension at least $2$. We can move the geodesic components of $c$ in their respective maximal flats without changing the homotopy class of $c$. Together with uniform separation of geodesics this yields the uniform separation of the flats supporting $c$. In the general case this argument is replaced by Lemma \ref{lrepelling}. The flats supporting $c$ are uniformly separated, say $\delta$-separated so their $\delta/2$-thickenings are disjoint. Since the injectivity radius is bounded from below by $R_n$ we can show that the $\delta/2$-thickening of a flat containing a closed geodesic $\gamma$ of length $\ell(\gamma)$ has volume $\gg\ell(\gamma)R_n$. This means that the total length of a reduced representative is bounded, up to a constant, by $\Vol(\Gamma_n\bs X)/R_n$.

In the general case the sequence $(\Gamma_n\bs X)$ converges to $X$ in the Benjamini-Schramm topology \cite{7Sam} so the injectivity radius in a typical point is large. We show that to each geodesic supporting $c$ we can attach a "local flat" of large volume and that those local flats are disjoint. We then deduce that $c$ cannot fill $\Gamma_n\bs X$ with positive density so $\ell(c)=o(\Vol(\Gamma_n\bs X))$.
\subsection{Structure of the article} In Section \ref{S1} we establish the connection between the lengths of representatives of the homology classes and the dimension of the first homology group. The main tools are the simplicial complexes constructed by Gelander in \cite{Gelander1,Gelander2} and the Benjamini-Schramm convergence of higher rank locally symmetric spaces established in \cite{7Sam}. In Section \ref{S2} we prove that in a higher rank locally symmetric space $M$ all the homology classes in $H_1(M,\F_2)$ are represented by a cycle whose length inside the thick part is sublinear in the volume of $M$ from which we deduce the archimedean case in Theorem \ref{tmainL}. We also give the proof of the second main theorem (Theorem \ref{tmain}). In the last Section \ref{S3} we indicate how to adapt the proofs to the non-Archimedean setting of the Bruhat-Tits buildings.
\subsection*{Acknowledgement} Part of this work was done as a part of the author's thesis at Universit\'e Paris Sud supported by a public grant as part of the Investissement d'avenir project, reference ANR-11-LABX-0056-LMH, LabEx LMH. The author was partially supported by the ERC Consolidator Grant No. 648017.
\section{Lengths of homology classes and the dimension}\label{S1}
Let $X$ be  a symmetric space of higher rank such that $G={\rm Isom(X)}$ has property $(T)$. Let $\Gamma$ be a torsion-free lattice in $G$ and fix $p=2$. In the sequel write $M=M_{\Gamma}:=\Gamma\bs X$ if $\Gamma$ is uniform or $M_\Gamma:=(\Gamma\bs X)_{\leq \varepsilon}$ if $\Gamma$ is non-uniform\footnote{We perform this modification because in Lemma \ref{lexistence} we need $M$ to be compact.} and $\varepsilon<1$ is small enough so that $\Gamma\bs X$ retracts onto $(\Gamma\bs X)_{\leq\varepsilon}$ and $\Vol(\Gamma\bs X)-\Vol(M)\geq 1$. Note that this means the bounds in terms of $\Vol(M)$ translate into the same bounds in terms of $\Vol(\Gamma\bs X)$. We also have $H_1(\Gamma\bs X,\F_p)\simeq H_1(M,\F_p)$. The reasoning is carried out for any prime $p$, the special properties of $p=2$ play an important part only in the next section. We write $Z_1(M,\F_p)$ for the module of $1$-cycles on $M$ with coefficients in $\F_p$. Any cycle $c\in Z_1(M,\F_p)$ can be represented as 
$$c=\sum_{i\in I}a_i\gamma_i,$$
where $I$ is a finite set of indices, $a_i\in \F_p\setminus\{0\}$ and $\gamma_i$ are oriented smooth differentiable curves $\gamma_i\colon [0,1]\to M$. Fix $R>0$, the $R$-length of a cycle $c$ is defined as \begin{equation}\label{eRlength}\ell^R(c):=\sum_{i\in I}\ell^R(\gamma),
\end{equation} where $\ell^R(\gamma_i)$ stands for the length of $\gamma_i\cap M_{>R}$. We put $\ell(c):=\ell^0(c)$. We define the \textbf{total length} (resp. total $R$-length) of $\alpha\in H_1(M,\F_p)$ by 
\begin{align}
\ell(\alpha)=&\inf_{\substack{c\in Z_1(M,\F_p)\\ [c]=\alpha}}\ell(c)\\
\ell^R(\alpha)=&\inf_{\substack{c\in Z_1(M,\F_p)\\ [c]=\alpha}}\ell^R(c)
\end{align}
Finally, the normalized \textbf{$R$-length} of $M$ is defined as $$\ell^R(M)=\frac{\sup_{\alpha_\in H_1(M,\F_p)}\ell^R(\alpha)}{\Vol(M)}.$$ The following proposition is the main goal of this section:
\begin{proposition}\label{p.MTech} For every $\delta>0$ there exist $\delta'>0$ such  that for every locally symmetric space $M=\Gamma\bs X$ with $\ell^R(M)\leq {\delta'}$ and $\Vol(M)$ big enough we have
$$\dim_{\F_p}H_1(M,\F_p)\leq \delta\Vol(M).$$
\end{proposition}
We shall prove it after introducing some tools. In \cite{Gelander1,Gelander2} Gelander constructed for every manifold $M=\Gamma\bs X$ a simplicial complex $\mathcal N$, with $\pi_1(\mathcal N)\simeq \pi_1(M)$, with the number of vertices bounded by $A\Vol(M)$ and degrees bounded 
uniformly by $B$ for certain constants $A=A(X),B=B(X)$ dependent only on $X$. We shall extract from his construction the following lemma:
\begin{lemma}\label{l-cycleapprox}
Let $M,X$ be as before and let $\mathcal N$ be a simplicial complex constructed in \cite{Gelander2} and let $R>0$ be bigger than the Margulis constant of $X$. There exists a constant $C_1=C_1(X)$ such that any homology class $\alpha\in H_1(M,\F_p)$ is represented by an integral combination $\sum_{i\in I}a_i e_i,$ where $e_1$ are edges of $\mathcal{N}$ and $|I|\leq C_1\ell^R(\alpha) + O(\Vol(M_{<R}))$.
\end{lemma}
\begin{proof}
Let us recall few details of the Gelander's construction \cite{Gelander2}. Write $d:X\times X\to X$ for the Riemannian metric on $X$ and let $\varepsilon_G,\mu$ be as in \cite[p.5-6]{Gelander2}. 
Let $\varepsilon=\varepsilon_G/2$. Write $M_{\geq R}$ for the $R$-thick part and $M_{\leq}$ for the $R$-thin part of $M$. In \cite{Gelander2}, Gelander shows that there exists a closed submanifold $N$ of codimension at least $2$ and a subset $\psi_{\leq 0}$ of $M\setminus N$ such that $M\setminus N$ retracts onto arbitrarily small neighborhood of $\psi_{\leq 0}$. By construction the subset $\psi_{\leq 0}$ is contained in $M_{\geq \frac{\varepsilon}{2\mu}}$. The inclusion $\psi_{\leq 0}\subset M\setminus N$ induces a surjective map $i^*\colon\pi_1(\psi_{\leq 0})\to\pi_1(M)$. Let $\kappa=\frac{\varepsilon}{2\mu}$ (in \cite{Gelander2} this quantity is called $\alpha$, we change the notation to avoid confusion with homology classes). We choose a maximal $\kappa/2$-separated subset $\mathcal S$ of $\psi_{\leq 0}$. The union of $\kappa$-balls around the points of $\mathcal S$ covers $\psi_{\leq 0}$. We denote the covering by $\mathcal{U}$. Write $\mathcal U_{\leq R}$ for the subset of those sets in $\mathcal U$ that are contained in $M_{\leq R}$.  The complex $\mathcal N$ is obtained as the nerve of $\mathcal U$ and by general theory it is homotopic to the union of sets in $\mathcal U$. In particular $\pi_1(\mathcal N)$ surjects onto $\pi_1(\psi_{\leq 0})$ which surjects onto $\pi_1(M)$. Because of this we justified in saying that homology classes in $H_1(M,\F_p)$ are represented by cycles in $Z_1(\mathcal{N},\F_p)$. 

 Define the subcomplex $\mathcal N_{\leq R}$ as the nerve of $U_{\leq R}$, it is a subcomplex of $\mathcal N$ which has only $O(\Vol(M_{\leq R }))$ simplices.
  We are going to use the fact that both complexes are nerves of coverings by balls od radius $\kappa$. We are ready to prove  Lemma \ref{l-cycleapprox}:

\textbf{Step 1.} Let $\gamma$ be a simple closed geodesic on $M$. Write $\gamma=\gamma_1\sqcup\gamma_2$ where $\gamma_1=\gamma\cap M_{\leq R}$ and $\gamma_2=\gamma\cap M_{\geq R}$. By perturbing $\gamma$ by an arbitrarily small amount we can assume it is disjoint from $N$. Write $r:M\setminus N\to \psi_{\leq 0}$ for the retract defined by Gelander. Then $i^*(r(\gamma))$ represents the same homotopy class as $\gamma$. Note that by the formulas defining $\psi$ \cite[p.7]{Gelander2}if $R\geq 2\varepsilon$ then $\psi_{\leq 0}\subset M_{\geq R}$. In particular $r$ is identity on $M_{\geq R}$. We have $r(\gamma)=r(\gamma_1)\cup\gamma_2$. Now, as $\gamma_1$ passes through $\psi_{\leq 0}$ and $\gamma_2 \subset M_{\geq R}$ we can find a finite families of balls $\mathcal F_i, i=1,2$ from the good cover $\mathcal U$ such that $\gamma_i\subset \bigcup_{U\in \mathcal F_i}U,|\mathcal{F}_1|=O(\Vol(M_{<R}))$ and $|\mathcal{F}_3|\leq C_1 \ell(\gamma_2)$. Set $\mathcal F_2$ can be taken as the set of all balls in $\mathcal U$ intersecting $\gamma_2$. The centers of balls in $\mathcal U$ are uniformly separated by $\kappa/2$ , hence the inequality $|\mathcal{F}_2|\leq C_1 \ell(\gamma_2)$. We deduce that the homology class of $\gamma$ can be represented by a sum of certain number of edges in $\mathcal{N}_{\leq R}$, and at most $C_1\ell(\gamma_2)\leq C_1\ell^R(\gamma)$ edges from $\mathcal{N}$.

\textbf{Step 2.} Pick $\delta>0$ small. Let $c=\sum_{j\in J}a_j\gamma_j$ be a representative of $\alpha$ such that $\ell(c)^R\leq \ell^R(\alpha)+\delta$. By the first step we can represent $c$ as $c=c_1+c_2$ where 
$c_1=\sum_{e\in E_1\subset \mathcal{N}_{<R}}a_e e$ and $c_2=\sum_{e\in E_2\subset  \mathcal{N}}b_e e$ with $|E_2|\leq C_1\ell^R(\alpha)+C_1\delta$. The number of edges in $\mathcal{N}_{\leq R}$ is $O(\Vol(M_{\leq R})$ so $|E_1|=O(\Vol(M_{\leq R}))$. We put the inequalities together to get that $\alpha$ is represented by a combination of $C_1\ell^R(\alpha)+O(\Vol(M_{\leq R}))$ edges. 
\end{proof}

\begin{lemma} \label{lestimate} For every $0<\delta<\frac{1}{2}$ and $n$ big enough we have
$$\sum_{i=1}^{[\delta n]}{n\choose i}(p-1)^i\ll p^{\delta(3-\log\delta) n}$$
\end{lemma}
\begin{proof} By Stirling approximation
\begin{align}
\sum_{i=1}^{[\delta n]}{n\choose i}(p-1)^i &\leq (p-1)^{\delta n}\delta n{n\choose [\delta n]}\\
&\ll (p-1)^{\delta n}\delta n\frac{n^{\delta n}e^{\delta n}}{(\delta n)^{\delta n}}\ll (p-\frac{1}{2})^{\delta n}\left(\frac{e^\delta}{\delta^\delta}\right)^n\\
&= p^{\delta((\log p)^{-1}-\log\delta+1) n}\leq p^{\delta(3-\log\delta)n}.  
\end{align}
\end{proof}

\begin{proof}[Proof of Proposition \ref{p.MTech}]
Let $\mathcal N$ be the simplicial complex constructed by Gelander, with the property that $\pi_1(\mathcal N)\simeq \pi_1(M)$. Recall that there are constants $A,B$ dependent only on the symmetric space $X$ such that $\mathcal N$ has at most $A\Vol(M)$ vertices, all with degrees bounded by $B$. Let $C_1$ be as in Lemma \ref{l-cycleapprox}, let $\delta>\frac{1}{2}$ and let $\delta'>0$ be such that if we put $\delta''=2C_1\delta'/AB$ then $\delta''(3-\log \delta'')\leq \delta/2ABC_1$. Assume that $\ell^R(M)\leq \delta'$. By Lemma \ref{l-cycleapprox} every class in $H_1(M,\F_p)$ can be represented as a sum of at most $C_1\delta'\Vol(M)+ O(\Vol(M_{\leq R}))$  1-cells in $\mathcal N$. By \cite[Theorem 1.5]{7Sam} we have $\Vol(M_{\leq R})=o(\Vol(M))$. This is the only place where we use the assumption that ${\rm Isom}(X)$ has property $(T)$. Hence, for big enough $\Vol(M)$ every class in $H_1(M,\F_p)$ is represented by a sum of at most $2C_1\delta'\Vol(M)$ 1-cells in $\mathcal N$. Total number of $1$-cells in $\mathcal N$ is bounded by $AB\Vol(M)$. Applying lemma Lemma \ref{lestimate} with $n=AB\Vol(M)$ and $\delta''$ we deduce that the number of such representatives is bounded by $p^{\delta \Vol(M)}$. We infer that for $\Vol(M)$ big enough we have $\dim_{\F_p}H_1(M,\F_p)\leq\delta \Vol(M)$.  
\end{proof}
\section{Reduced representatives}\label{S2}
 The aim of this section is to show that the assumptions of Proposition \ref{p.MTech} are automatically satisfied once $\Vol(M)$ and $R$ are big enough:
\begin{proposition}\label{plength} There exists a positive constant $C_2$ such that for any $R>1$ and $Vol(M)$ big enough (depending on $R$) the following holds. Let $\alpha\in H_1(M,\F_2)$ and let $c\in Z_1(M,\F_2)$ be a reduced representative of $\alpha$. Then $\ell^R(c)\leq C_2\Vol(M)R^{-1/2}$. In particular, for $\Vol(M)$ big enough we have
$$\ell^R(M)\leq C_2 R^{-1/2}.$$
\end{proposition} We postpone the proof until the end of this section. Once we have Proposition \ref{plength} our main results are simple consequences: 
\begin{proof}[Proof of Theorem \ref{tmainL}]
Let $R>1$. It will be enough to prove that there exists a constant $C_3>0$ independent of $R$ such that for  $\Vol(M)$ big enough we have $\ell(\alpha)\leq C_3 \Vol(M)R^{-1/2}$ for any homology class $\alpha\in H_1(M,\F_2)$. To prove this we will use Lemma \ref{l-cycleapprox} to upgrade the inequality on $\ell^R(\alpha)$ to an inequality on $\ell(\alpha)$.  
Let $\mathcal N$ be a Gelander complex for $M$. It is constructed as a nerve complex of a cover $\mathcal U$ of a subset of $M$ by balls of radius $\kappa$. In the construction \cite{Gelander2} the constant $\kappa$ (denoted there as $\alpha$) depends only on $X$. The 1-skeleton $\mathcal N^1$ has a natural graph structure $\mathcal N^1=(V,E)$ whith vertex set $V=\{v_U| U\in \mathcal U\}$ indexed by the opens set in the good cover $\mathcal U$  and the edge set $E=\{e_{U,V}\mid U,V\in\mathcal{U}, U\cap V\neq \emptyset\}$. We have
$$\mathcal{N}^1=\left(\bigsqcup_{U} v_U\sqcup \bigsqcup_{U\cap V\neq \emptyset} [0,1]\times e_{U,V}\right)/\sim, $$ where $\sim$ is the equivalence relation given by $v_U\sim (0,e_{U,V}), v_V\sim (1,e_{U,V})$ and $(t,e_{U,V})\sim (1-t,e_{V,U})$ for $t\in [0,1]$. We endow $\mathcal{N}^1$ with the graph metric where edges are of length $1$. We are going to construct an explicit $2\kappa$-Lipschitz map $\tau\colon\mathcal N^1\to \Gamma\bs X$ inducing a surjective map on the fundamental groups.
For every set $U\in \mathcal U$ choose a point $x_U\in U$ and for every pair $U,V$ with non trivial intersection we choose a path $p_{U,V}:[0,1]\to U\cup V$ connecting $x_U$ and $x_V$. The diameters of the sets in $\mathcal U$ are bounded by $\kappa$ so we can choose patches $p_{U,V}$ of length less than $2\kappa$. We define a map $\tau:\mathcal N^1\to M$ by putting $\tau(v_U)=x_U$ and $\tau(t,e_{U,V})=p_{U,V}(t)$. This map is $2\kappa$-Lipschitz and $\pi_1(\tau)\colon \pi_1(\mathcal{N}^1)\to \pi_1(\Gamma\bs X)$ is surjective.

By Lemma \ref{l-cycleapprox} and Proposition \ref{plength}, if $\Vol(M)$ is big enough then the class $\alpha\in H^1(M,\F_2)$ is represented by an integral combination $c_0:=\sum_{i\in I}a_i e_i,$ where $e_1$ are edges of $\mathcal{N}$ and $\sum_{i\in I}|a_i|\leq C_1C_2\Vol(\Gamma\bs X)R^{-1/2} + O(\Vol(M_{<R})$. The map $\tau$ is $2\kappa$-Lipshitz so the image  $\tau(c_o)\in Z_1(M,\F_2)$ satisfies $$\ell(\tau(c_0))\leq 2\kappa C_1C_2\Vol(M)R^{-1/2} +O(\Vol(\Gamma\bs X)_{<R}).$$ By \cite[Theorem 1.5]{7SamN} $\Vol(M_{<R})=o(\Vol(M))$ so for $\Vol(M)$ big enough we will have $\ell(\tau(c_0))\leq C_3\Vol(\Gamma\bs X)R^{-1/2}$ with $C_3=2\kappa(C_1C_2+1)$.

\end{proof}
\begin{proof}[Proof of Theorem \ref{tmain}]
Let $\delta>0$. By Proposition \ref{p.MTech} there exists $\delta'$ such that $\dim_{\F_2}H_1(M,\F_2)\leq \delta \Vol(M)$ for $\Vol(M)$ big enough and $M$ such that $\ell^R(M)\leq \delta' \Vol(M)$. Pick $R\geq (\delta')^{-2}C_3^2$. By Proposition \ref{plength} we have $\ell^R(M)\leq \delta'$ so 
$$\limsup_{\Vol(M)\to\infty} \frac{ \dim_{\F_2}H_1(M,\F_2)}{\Vol(M)}\leq \delta.$$
To get the Theorem we let $\delta$ go to $0$.
\end{proof}  
Fix $R>1$. Recall that a reduced representative of a homology class $\alpha\in H_1(M,\F_p)$ is a cycle $c\in Z_1(M,\Z)$ such that $\ell(c)=\ell(\alpha)$. Standard compactness argument yields 
\begin{lemma}\label{lexistence}
Every class $\alpha\in H_1(M,\F_2)$ has a reduced representative. It is an $\F_2$-combination of closed geodesics. In general it is not unique.
\end{lemma}
\begin{remark} We could define reduced representatives with respect to the $R$-length but as it will turn out, any reduced representative $c$ already satisfies $\ell^R(c)=o(Vol(M))$. This is already enough to show that $\ell^R(M)$ tends to $0$ as $\Vol(M)\to\infty$.
\end{remark}
From now on it will be important that we work with $p=2$. Being a reduced representative of a mod$-2$ homology class forces strong geometric constraints on $c$.  The following Lemma guarantees that whenever a cycle $c$ has two geodesic components that are not $\kappa_1-$separated in the thick part $M_{>R}$ then there is a mod$-2$ homologous cycle $c'$ with $\ell(c')\leq \ell(c)-\kappa_2$ for some positive constant $\kappa_2$.  We will write $[a,b]$ for the shortest geodesic connecting points $a$ and $b$ and $B_M(x,r),S_M(x,r)$ for the ball and sphere of radius $r$ around $x$.
\begin{lemma}\label{lseparation}
There exist $\kappa_1,\kappa_2>0$ with the following property. For any two closed, non-contractible curves $\gamma_1,\gamma_2$ on $M$ such that $d_{M_{\geq R}}(\gamma_1,\gamma_2)\leq \kappa_1$ there exists a cycle $c\in Z_1(M,\F_2)$ such that $\ell(c)\leq \ell(\gamma_1)+\ell(\gamma_2)-\kappa_2$ and $[c]= [\gamma_1+\gamma_2]$ in $H_1(M,\F_2)$. 
\end{lemma}
\begin{proof}
Let $x_1,x_2$ be points on $\gamma_1\cap M_{\geq R},\gamma_2\cap M_{\geq R}$ respectively, such that $d(\gamma_1,\gamma_2)=d(x_1,x_2)$. Let $y$ be the midpoint of the shortest path connecting $x_1,x_2$ in $M_{\geq R}$. Fix some radius $R'<1$ and consider the intersection of $\gamma_1,\gamma_2$ with $B_M(y,R')$. Note that $R'<R$ so the ball $B_M(y,R')$ is isometric to an $R'$-ball in $X$. Since neither of $\gamma_1,\gamma_2$ is contractible, they have non-empty intersection with $S_M(y,R')$. Let $p_i,q_i$ be intersection points of $\gamma_i$ with the sphere $S_M(y,R')$ such that $x_i$ lies on the segment of $\gamma_i$ bounded by $p_i$ and $q_i$, for $i=1,2$. 
\begin{figure}[H]
\includegraphics[height=2.2in]{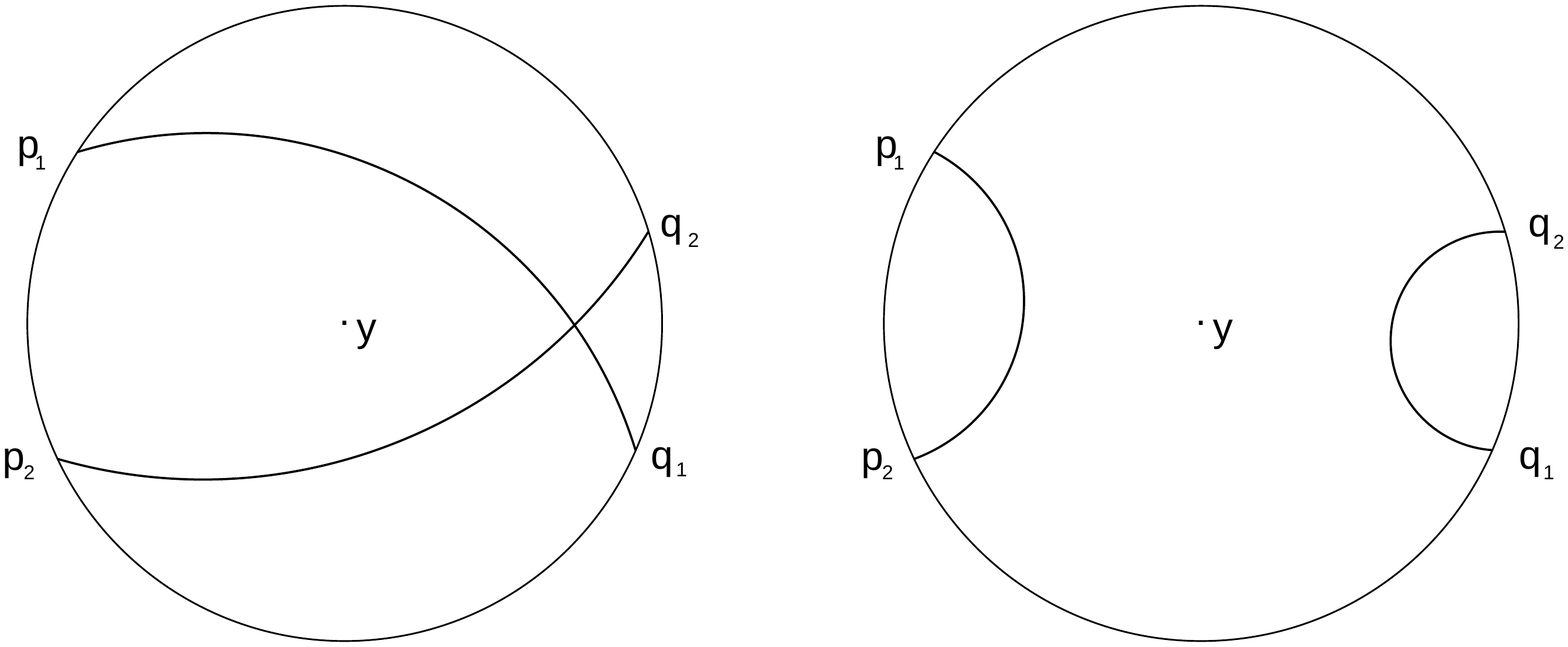}
\caption{Before (left) and after (right) a surgery on a pair of close curves in $B_M(y,R')$.}\label{fig1}  
\end{figure}

To construct the new cycle $c$ we replace segments $\gamma_1\cap B(y,R')$ and $\gamma_2\cap B(y,R')$ by geodesics $[p_1,p_2],[q_1,q_2]$ or $[p_1,q_2],[p_2,q_1]$ as shown in Figure \ref{fig1}. We always choose the pair with minimal total length. For $R'$ small enough (how small depending only on $X$) the metric inside $B(y,R')$ is close to the flat euclidean metric so for $\kappa_1$ close to $0$ it is evident that one of those operations will reduce the total length by at least $\kappa_2$ for some positive constant $\kappa_2$. Note that this surgery does not change the mod$-2$ homology class.
\end{proof}
The second lemma says that in higher rank we have a lot of freedom to perturb closed geodesics into curves with a minimal increase in length. From this point onward we assume for technical reasons that $R>2(1+\kappa_1+\kappa_2)$. This is not a problem since later we are going to let $R$ tend to infinity.
\begin{lemma}\label{lrepelling}
Let $\gamma$ be a closed geodesic on $M$. Let $\kappa_1,\kappa_2$ be as in Lemma \ref{lseparation}. Put 
$$N(\gamma)=\{x\in M_{\geq R}\mid \exists\gamma' \textrm{ curve homotopic to } \gamma \textrm{ such that } d(x,\gamma')< \kappa_1/2\textrm{ and }\ell(\gamma')<\ell(\gamma)+\kappa_2/2\}.$$
Then $\Vol(N(\gamma))\geq C_3 \ell^R(\gamma)R^{\frac{1}{2}}$ for some absolute positive constant $C_3$. 
\end{lemma}
\begin{proof}Write $\iota\colon X\to M$ for the covering map and $B(\Sigma,\varepsilon)$ for the open $\varepsilon$-neigborhood of a set $\Sigma$.
Let $x_1,\ldots,x_N$ be a maximal $R$-separated subset of $\gamma\cap M_{>R}$. Clearly we have $N\geq \ell^R(\gamma)/2R$. Choose a lift $\tilde\gamma$ of $\gamma$ to $X$ and let $\tilde x_i$ be lifts of $x_i$'s  lying on $\tilde{\gamma}$. There exits a maximal flat $F$ of $X$ containing $\tilde\gamma$\footnote{Flat $F$ does not necessarily descend to a closed flat on $M$.}. Since $X$ is a higher rank symmetric space $F$ is isometric to $\R^d$ with $d=\rank X\geq 2$ \cite[p.152]{GeomGT-NRW}. For every $i=1,\ldots, N$ let $F_i=F\cap B(x_i,R-\kappa_1/2)$ and $G_i=B(F_i,\kappa_1/2)$. Note that covering map $\iota\colon X\to M$ is injective on $G_i$ for every $i$ and the images of $G_i$'s via $\iota$ in $M$ are pairwise disjoint. Let $\tilde x'_i,\tilde x''_i$ be the pair of points on $\tilde\gamma$ at distance $R-\kappa_1/2-\kappa_2$ from $\tilde x_i$. Consider the open ellipsoid $E_i$ in $F_i$ defined as $E_i=\{y\in F_i\mid d(y,\tilde x'_i)+d(y,\tilde x''_i)< 2R+\kappa_2/2\}$ (see Figure \ref{fig2}).
\begin{figure}[H]
\includegraphics[scale=0.35]{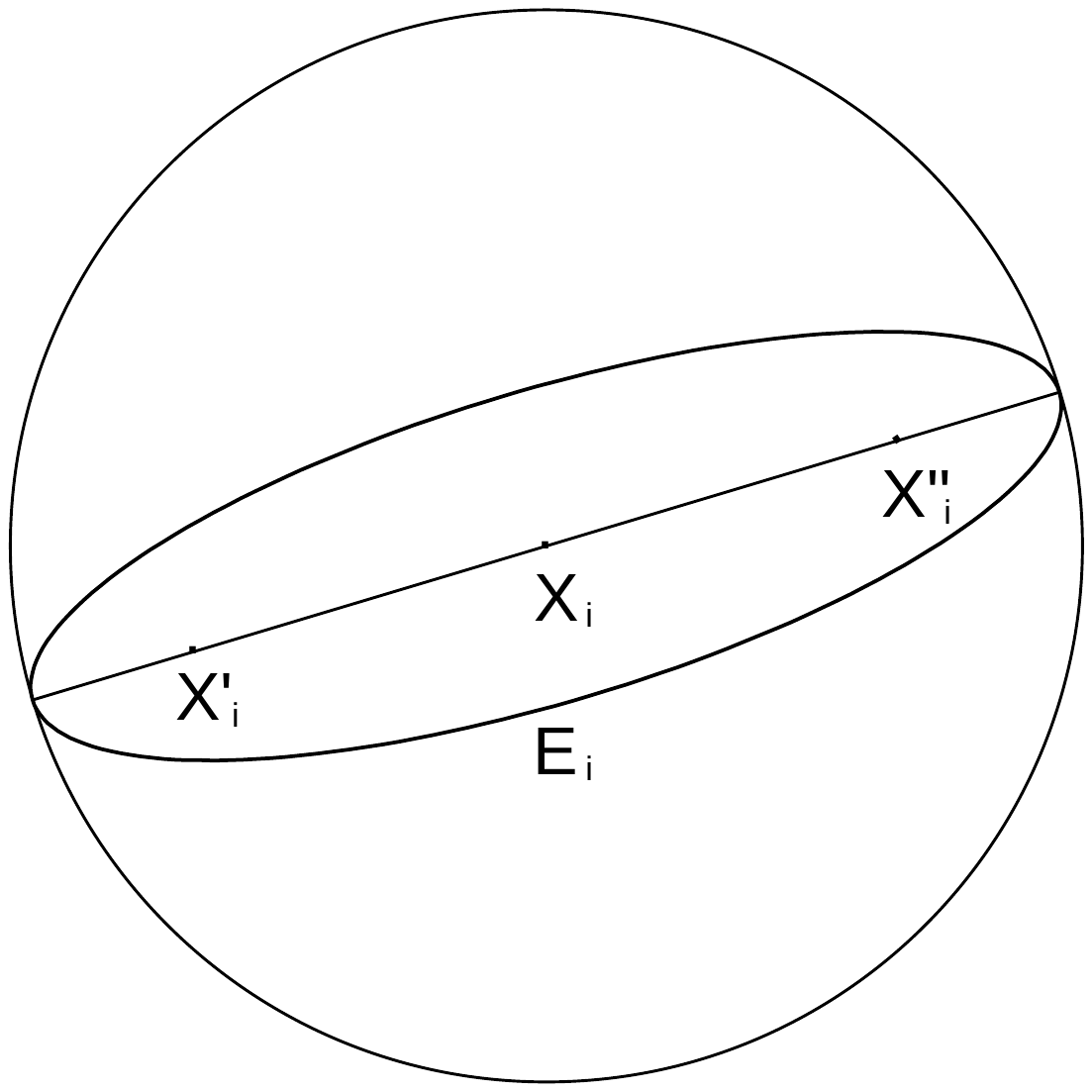}
\caption{Image of $E_i\subset F_i$ under the covering map.}\label{fig2}  
\end{figure} Note that for any point $y\in \iota(E_i)$  there exists a closed curve $\gamma'$ homotopic to $\gamma$ such that $\gamma'$ passes through $y$ and $\ell(\gamma')< \ell(\gamma)+\kappa_2/2$.  Put $H_i=B(E_i,\kappa_1/2)$. We have $H_i\subset G_i$ so the images of $H_i$ via $\iota$ are pairwise disjoint. Formula for the volume of an ellipsoid yields $\Vol(H_i)\geq C_4 R^{\frac{d+1}{2}}\geq C_4 R^{\frac{3}{2}}$ for some positive constant $C_4$ depending only on $X$ and $\kappa_1$. Hence $$\Vol(\bigsqcup_{i=1}^N H_i)\geq C_3\ell^R(\gamma)R^{\frac{1}{2}}.$$ By construction $\iota(H_i)\subset N(\gamma)$ for every $=1,\ldots, N$ which ends the proof.
\end{proof}
As a corollary of Lemmas \ref{lseparation} and \ref{lrepelling} we get
\begin{corollary}\label{cseparation}Let $\alpha\in H_1(M,\F_2)$ and let $c$ be a reduced representative of $\alpha$. Write 
$c=\sum_{i\in I}\gamma_i$. Then the sets $N(\gamma_i), i\in I$ defined as in Lemma \ref{lrepelling} are pairwise disjoint.
\end{corollary}
\begin{proof}
We argue by contradiction. Let $\gamma_1,\gamma_2$ be two geodesic components of $c$ such that $N(\gamma_1)\cap N(\gamma_2)\neq \emptyset$. Let $y\in N(\gamma_1)\cap N(\gamma_2)$. By definition $y\in M_{>R}$ and there exist closed curves $\gamma'_1,\gamma'_2$ homotopic to $\gamma_1,\gamma_2$ respectively such that $\ell(\gamma'_i)< \ell(\gamma_i)+\kappa_2/2$ and $d(\gamma'_i,y)<\kappa_1/2$ for $i=1,2$. Let $c'$ be the cycle obtained from $c$ by replacing $\gamma_i$ by $\gamma'_i$ for $i=1,2$. It represents the same homology class. We have $\ell(c')<\ell(c)+\kappa_2$. Curves $\gamma'_1,\gamma'_2$ satisfy $d_{M_{>R}}(\gamma'_1,\gamma'_2)<\kappa_1$ so we may perform the surgery from Lemma \ref{lseparation} to construct homologous cycle $c''$ with $\ell(c'')\leq \ell(c')-\kappa_2<\ell(c)$. The last inequality contradicts the assumption that $c$ was a reduced representative.
\end{proof}

\begin{proof}[Proof of Proposition \ref{plength}]
Let $\alpha\in H_1(M,\F_2)$ and let $c$ be a reduced representative of $\alpha$. Then 
$c=\sum_{i\in I}\gamma_i$ is a combination of closed geodesics indexed by some set $I$. By Corollary \ref{cseparation} the sets $N(\gamma_i), i\in I$ are pairwise disjoint and by Lemma \ref{lrepelling} we have $\Vol(N(\gamma_i))\geq C_3\ell^R(\gamma_i)R^{\frac{1}{2}}$ so $$\ell^R(c)=\sum_{i\in I}\ell^R(\gamma_i)\leq \frac{\Vol(M_{>R})}{C_3 R^{\frac{1}{2}}}.$$ We deduce that $\ell^R(M)\leq C_3^{-1}R^{-\frac{1}{2}}$.
\end{proof}
\section{The non-archimedean case}\label{S3}
Let $X=$ be the Bruhat-Tits building of a semisimple algebraic group $G$ defined over a non-archimedean local field $F$. We say that $X$ is higher rank if the dimension of $X$ is at least $2$ (equivalently the $F$-rank of $G$ is at least $2$). In this section we adapt the proof of the archimedean case of Theorem \ref{tmain} to show the following:
\begin{theorem}\label{t.mNonarch}
Let $X$ be a higher rank building and let $(\Gamma_n)_{n\in\N}$ be a sequence of irreducible, torsion-free lattices in $G$ such that $(\Gamma_n\bs X)_{n\in\N}$ converges to $X$ in Benjamini-Schramm topology. Then 
$$\lim_{n\to\infty}\frac{\dim_{\F_2}H_1(\Gamma_n\bs X,\F_2)}{\Vol(\Gamma_n\bs X)}=0.$$
\end{theorem}
We remark that recently Gelander and Levit \cite{LevitGelander} proved that if $G$ is higher rank non-archimedean semisimple algebraic group with property $(T)$ then every non-stationary sequence complexes $(\Gamma_n\bs X)_n$ converges Benjamini-Schramm to $X$, so the non-Archimedean case of Theorem \ref{tmain} follows.
 
Let us indicate why the proof of the archimedean case does not work for buildings. The reason for this is that Lemma \ref{lseparation} fails in the non-archimedean setting. Consider the following example: let $X$ be the building of $\SL(3,\Q_p)$ and let $A_1,A_2$ be two apartments positioned with respect to each other so that $\gamma:=A_1\cap A_2$ is an infinite geodesic. Pick a point $p$ on $\gamma$ and let $c_i$ (for $i=1,2$) be the geodesic in $A_i$ passing through $p$ and orthogonal to $\gamma$. Even though $c_1,c_2$ intersect non-trivially in $p$
there is no way of cutting and reconnecting $c_1$ and $c_2$ which locally reduces the length. To deal with this issue we will need another version of Lemma \ref{lseparation} which will tell us that if a large number of geodesics intersect a small ball, then we can choose two for which there is a way of cutting and reconnecting that reduces total length by at least $\kappa_2>0$. Let $M=\Gamma\bs X$ (this time the quotient is always compact). We adopt all the notation from the previous sections.
\begin{lemma}\label{lseparation2}
Fix $R\geq 2$. There exist constants $C_4,\kappa_1,\kappa_2>0$ with the following property. Let $x\in M_{>R}$ and let $\gamma_1,\ldots \gamma_m, m\geq C_4$ be a collection of curves passing through $B(x,\kappa_1/2)$ such that $\gamma_i\cap B(x,2)$ is geodesic for $i=1,\ldots, m$. Then there exists $1\leq i<j\leq m$ and a cycle $c_{i,j}\in Z_1(M,\F_2)$ such that $\ell(c_{i,j})\leq \ell(\gamma_i)+\ell(\gamma_j)-\kappa_2$ and $[c_{i,j}]= [\gamma_i+\gamma_j]$ in $H_1(M,\F_2)$. 
\end{lemma}
\begin{proof}
 There is only finitely many ways the apartments in $X$ can look like in a ball of radius $2$. The idea is to use this observation and apply the pigeonhole principle to reduce the proof of Lemma \ref{lseparation2} to the archimedean case. 

Let $\kappa_1,\kappa_2$ be as in Lemma \ref{lseparation}. For a technical reason we want $\kappa_1<1$. This is not a problem since we can always take smaller $\kappa_1$. Choose a lift $\tilde x\in X$ of $x$. We have an isometry $B_M(x,2)\simeq B_X(\tilde x,2)$. Let $\tilde\gamma_i$ be the lift of $\gamma_i$ which intersects $B_X(\tilde x,\kappa_1/2)$. From this point onward we forget about $M$ and work in $B(\tilde x,2)$. Every manipulation we do in this ball descends to $M$. For $i=1,\ldots,m$ choose an apartment $A_i$ containing $\tilde\gamma_i$ and let $L_i=A_i \cap B_X(\tilde x,1)$. Since $X$ is a (poly)simplicial complex of bounded geometry there is a uniform upper bound $C_4=C_4(X)$ on the number of possible intersections of an apartment with $B(x,2)$ so the number of distinct $L_i$ is bounded by $C_4$. Since $m>C_4$, by pigeonhole principle there are $1\leq i<j\leq m$ such that $L_j=L_i:=L$. The set $L$ is isometric to an Euclidean ball of dimension at least $2$ with radius at least $2-\kappa_1/2\geq 1$. Curves $\tilde\gamma_i\cap L, \tilde\gamma_j\cap L$ are straight lines passing through $B_L(\tilde{x},\kappa_1/2)$ so we can repeat the same construction as in the proof of Lemma \ref{lseparation}.
\end{proof}
To apply Lemma \ref{lseparation2} we will need a bit more technical version of Lemma \ref{lrepelling}. 
\begin{lemma}\label{lrepelling2}
Let $\gamma$ be a closed geodesic on $M$. Let $\kappa_1,\kappa_2$ be as in Lemma \ref{lseparation2}. Put 
\begin{align*}N'(\gamma)=&\{x\in M_{\geq R}\mid \exists\gamma' \textrm{ curve homotopic to } \gamma \textrm{ such that } d(x,\gamma')< \kappa_1/2,  \textrm{ the curve } \gamma'\cap B(x,2+\kappa_1) \\ &\textrm{ is geodesic and }\ell(\gamma')<\ell(\gamma)+\kappa_2/2\}.\end{align*}
Then $\Vol(N'(\gamma))\geq C_5 \ell^R(\gamma)R^{\frac{1}{2}}$ for some absolute positive constant $C_5$. 
\end{lemma} The proof is virtually the same construction as in Lemma \ref{lrepelling} so we skip it.
Proposition \ref{p.MTech} remains true for buildings. The proof in that case in even easier than for locally symmetric spaces because the quotients of buildings come with a structure of (poly)simplicial complexes and their injectivity radius is uniformly bounded from below.

We may now proceed to the proof of Theorem \ref{t.mNonarch}. Let $\alpha\in H_1(M,\F_2)$ and let $c\in Z_1(M,\F_2)$ be a representative of the minimal $R$-length. Write $c=\sum_{i=1}^k\gamma_k$ where $\gamma_k$ are closed geodesics. We claim that every point $x\in M_{>R}$ is contained in at most $C$ sets $N'(\gamma_i)$ for $i=1,\ldots, k$. Indeed, if is not the case we then up to change in the enumeration we can assume that $x\in N'(\gamma_{i})$ for $i=1,\ldots,m$. Let $\gamma'_{i}$ be a curve homotopic to $\gamma_{i}$  such that $d(x,\gamma'_i)< \kappa_1/2$,$\gamma'_i\cap B(x,2+\kappa_1)$ is geodesic and $\ell(\gamma'_i)<\ell(\gamma_i)+\kappa_2/2$.
By Lemma \ref{lseparation2} there are $i<j$ and a cycle $c_{i,j}\in Z_1(M,\F_2)$ such that $\ell(c_{i,j})\leq \ell(\gamma_i)+\ell(\gamma_j)-\kappa_2$ and $[c_{i,j}]= [\gamma_i+\gamma_j]$ in $H_1(M,\F_2)$. The cycle $c'=\sum_{l\neq i,j}\gamma_l+ c_{i,j}$ would be a representative of $\alpha$ with $\ell^R(c')<\ell^R(c)$. This contradicts the choice of $c$ and proves the claim.
It follows that 
\begin{align*}{C_4}{\Vol(M_{>R})}&\geq C_5 \sum_{i=1}^k R^{1/2}\ell^R(\gamma_i)=C_5\ell^R(c).\\
\frac{C_4\Vol(M_{>R})}{C_5R^{1/2}}&\geq \ell^R(c).
\end{align*} 
This proves Proposition \ref{plength} for the non-archimedean quotients with $C_2=C_5/C_4$. In a Benjamini-Schramm convergent sequence of quotients $\Gamma_n\bs X$ we have $\lim_{n\to\infty}\frac{\Vol((\Gamma_n\bs X)_{>R})}{\Vol(\Gamma_n\bs X)}=1$ so we can apply Proposition \ref{p.MTech} to get Theorem \ref{t.mNonarch}. 

With the non-archimedean version od Proposition \ref{plength} at our disposal the proof of the non-archimedean case of Theorem \ref{tmainL} is completely analogous to the proof in archimedean case. We just need to repeat the argument with $\mathcal N^1$ replaced by the $1$-skeleton of $\Gamma\bs X$.
\begin{remark} Most of our methods also apply to the quotients of products of symmetric spaces and Bruhat-Tits buildings by irreducible lattices. The only missing ingredient is the analogue of Gelander's construction of simplicial complexes homotopy equivalent to the thick part.  
\end{remark}
\bibliographystyle{plain}
\bibliography{../ref}
\end{document}